\documentclass[10pt, notitlepage]{article}   

\usepackage{amsmath,amsthm,amsfonts}   

\usepackage{amscd}
\usepackage{amsfonts}
\usepackage{amssymb}
\usepackage{color}      
\usepackage{epsfig}
\usepackage{graphicx}           

\theoremstyle{plain}
\newtheorem{Thm}{Theorem}[section]
\newtheorem{Lem}[Thm]{Lemma}
\newtheorem{Crlr}[Thm]{Corollary}
\newtheorem{Prop}[Thm]{Proposition}
\newtheorem{Obs}[Thm]{Observation}

\theoremstyle{definition}
\newtheorem{Def}[Thm]{Definition}

\theoremstyle{remark}

\def\finf{\mathop{{\rm I}\kern -.27 em {\rm F}}\nolimits}

\begin{document}

\title{\bf Domination Value in $P_2 \square P_n$ and $P_2 \square C_n$}

\author{\Large Eunjeong Yi\\
Texas A\&M University at Galveston\\
Galveston, TX 77553, USA\\
{\tt yie@tamug.edu}}

\maketitle

\date{}

\begin{abstract}
A set $D \subseteq V(G)$ is a \emph{dominating set} of a graph $G$ if
every vertex of $G$ not in $D$ is adjacent to at least one vertex in $D$.
A \emph{minimum dominating set} of $G$, also called a $\gamma(G)$-set, 
is a dominating set of $G$ of minimum cardinality. For each vertex $v \in V(G)$, we define the
\emph{domination value} of $v$ to be the number of $\gamma(G)$-sets to which $v$ belongs. In this paper, we find the total number of minimum dominating sets and
characterize the domination values for $P_2 \square P_n$ and $P_2 \square C_n$.
\end{abstract}



\section{Introduction}

Let $G = (V(G),E(G))$ be a simple, undirected, and nontrivial
graph. For $S \subseteq V(G)$, we denote by $\langle S \rangle$ the subgraph of $G$ induced by $S$.
For a vertex $v \in V(G)$, the \emph{open neighborhood of $v$} is the set $N(v)=\{u \mid uv \in E(G)\}$, and the
\emph{closed neighborhood of $v$} is the set $N[v]=N(v) \cup
\{v\}$. For $S \subseteq V(G)$, the \emph{open neighborhood of $S$} is the set $N(S)=\cup_{v \in S} N(v)$ 
and the \emph{closed neighborhood of $S$} is the set $N[S]=N(S) \cup S$.

A set $D \subseteq V(G)$ is a \emph{dominating set} if $N[D]=V(G)$, and is a \emph{total dominating set} if $N(D)=V(G)$. The \emph{domination number of a graph $G$}, denoted by $\gamma(G)$, is the minimum of the cardinalities of all dominating sets of $G$. A \emph{minimum dominating set} of $G$, also called a $\gamma(G)$-set, 
is a dominating set of $G$ of minimum cardinality. For discussions on domination (resp. total domination) in graphs, see \cite{B1, B2, EJ, Dom1, Dom2, 
Ore} (resp. see \cite{EJ2, Dom1, Henning}).  Slater \cite{Slater} introduced the notion of the number of dominating sets of $G$, which he denoted by HED$(G)$ in honor of Steve Hedetniemi on the occasion of his 60th birthday; further, Slater used \#$\gamma(G)$ to denote the number of $\gamma(G)$-sets. Following \cite{Kang, YI}, we denote by $\tau(G)$ the total number of $\gamma(G)$-sets. 
For each vertex $v \in V(G)$, we define the \emph{domination value} of $v$ in $G$,
denoted by $DV_G(v)$, to be the number of $\gamma(G)$-sets to
which $v$ belongs; we often drop $G$ when ambiguity is not a
concern. Clearly, $0 \le DV_G(v) \le \tau(G)$ for any graph $G$
and for any vertex $v \in V(G)$. See \cite{YI} for an introductory discussion on
domination value in graphs and \cite{Kang} for an introductory discussion on total domination value in graphs.

The \emph{Cartesian product} of two graphs $G$ and $H$,
denoted by $G \square H$, is the graph with the vertex set $V(G)
\times V(H)$ such that $(u,v)$ is adjacent to $(u', v')$ if and
only if (i) $u=u'$ and $vv' \in E(H)$ or (ii) $v=v'$ and $uu' \in
E(G)$. For other graph theory terminology, refer to \cite{CZ}.

We denote by $P_n$ and $C_n$ the path and the cycle on $n$ vertices, respectively. In \cite{Kinch}, Jacobson and Kinch obtained the results on
$\gamma(P_m \square P_n)$ for $m=2,3,4$. Later, Hare developed an
algorithm to compute $\gamma(P_m \square P_n)$ and was able to
find expressions for $\gamma(P_m \square P_n)$ for a number of
different values of $m$ and $n$ (see \cite{Hare}). Chang and
Clark proved the formulas found by Hare for $\gamma(P_5 \square
P_n)$ and $\gamma(P_6 \square P_n)$ in \cite{5xn}. The complexity of determining $\gamma(P_m \square P_n)$ is open as of \cite{Hedetniemi}. In \cite{cycles}, Klav\v{z}ar and Seifter 
obtained results on $\gamma(C_m \square C_n)$ for $m=3,4,5$.

In section 2, we present relevant results from \cite{YI}. In 
sections 3 and 4, noting $\gamma(P_2 \square P_n) \neq \gamma(P_2 \square C_n)$ for $n \equiv 0 \pmod 4$, we investigate the total number of minimum dominating sets and the domination value for two classes of graphs, $P_2 \square P_n$ and $P_2 \square C_n$.


\section{Preliminaries and domination value in paths and cycles}

We first recall the following observations.

\begin{Obs} \cite{YI} \label{Obs1}
$\displaystyle \sum_{v \in V(G)} DV_G(v) = \tau(G) \cdot
\gamma(G)$
\end{Obs}

\begin{Obs} \cite{YI} \label{Obs2}
If there is an isomorphism of graphs carrying a vertex $v$ in $G$
to a vertex $v'$ in $G'$, then $DV_G(v)=DV_{G'}(v')$.
\end{Obs}

It is well known that $\gamma(P_n)=\gamma(C_n)=\lceil \frac{n}{3} \rceil$. If we let the vertices of the path $P_n$ be labeled $1$ through $n$
consecutively, then we have the following

\begin{Thm} \cite{YI}
For $n \ge 2$,
\begin{equation*} \tau(P_n) = \left\{
\begin{array}{lr}
\ 1 & \mbox{ if } n \equiv 0  \mbox{ (mod 3)} \ \\
\ n+ \frac{1}{2}\lfloor \frac{n}{3} \rfloor (\lfloor \frac{n}{3}\rfloor-1) & \mbox{ if } n \equiv 1  \mbox{ (mod 3)} \ \\
\ 2+ \lfloor \frac{n}{3} \rfloor & \mbox{ if } n \equiv 2 \mbox{
(mod 3)}.
\end{array} \right.
\end{equation*}
\end{Thm}

For the domination value of a vertex $v$ on $P_n$, by Observation
\ref{Obs2}, $DV(v)=DV(n+1-v)$ for $1 \le v \le n$. More
precisely, we have the classification results which follow.

\begin{Crlr} \cite{YI}
Let $v \in V(P_{3k})$, where $k \ge 1$. Then
\begin{equation*}
DV(v)= \left\{
\begin{array}{ll}
0 & \mbox{ if } v \equiv 0,1  \mbox{ (mod 3) } \\
1 & \mbox{ if } v \equiv 2  \mbox{ (mod 3) } .
\end{array} \right.
\end{equation*}
\end{Crlr}

\begin{Prop} \cite{YI}
Let $v \in V(P_{3k+1})$, where $k \ge 1$. Write $v=3q+r$, where
$q$ and $r$ are non-negative integers such that $0 \le r < 3$.
Then, noting $\tau(P_{3k+1})=\frac{1}{2}(k^2+5k+2)$, we have
\begin{equation*} DV(v)= \left\{
\begin{array}{ll}
\frac{1}{2}q(q+3) & \mbox{ if } v \equiv 0 \mbox{ (mod 3) } \ \\
(q+1)(k-q+1) & \mbox{ if } v \equiv 1 \mbox{ (mod 3) } \ \\
\frac{1}{2}(k-q)(k-q+3) & \mbox{ if } v \equiv 2 \mbox{ (mod 3) }.
\end{array} \right.
\end{equation*}
\end{Prop}

\begin{Prop} \cite{YI}
Let $v \in V(P_{3k+2})$, where $k \ge 0$. Write $v=3q+r$, where
$q$ and $r$ are non-negative integers such that $0 \le r < 3$.
Then, noting $\tau(P_{3k+2})=k+2$, we have
\begin{equation*}
DV(v)= \left\{
\begin{array}{ll}
0 & \mbox{ if } v \equiv 0 \mbox{ (mod 3) } \ \\
1+q & \mbox{ if } v \equiv 1 \mbox{ (mod 3) } \ \\
k+1-q & \mbox{ if } v \equiv 2 \mbox{ (mod 3) } .
\end{array} \right.
\end{equation*}
\end{Prop}

If we let the vertices of the cycle $C_n$ be labeled 1 though $n$ cyclically, then we have the following

\begin{Thm}\cite{YI}\label{thm on C}
For $n \ge 3$,
\begin{equation*}
\tau(C_n) = \left\{
\begin{array}{lr}
\ 3 & \mbox{ if } n \equiv 0  \mbox{ (mod 3)} \ \\
\ n(1+\frac{1}{2} \lfloor \frac{n}{3} \rfloor) & \mbox{ if } n \equiv 1 \mbox{ (mod 3)} \ \\
\ n & \mbox{ if } n \equiv 2  \mbox{ (mod 3)} .
\end{array} \right.
\end{equation*}
\end{Thm}

By Theorem \ref{thm on C}, Observation \ref{Obs1}, Observation \ref{Obs2}, and the vertex-transitivity of $C_n$, we have the following

\begin{Crlr}\cite{YI}
Let $v\in V(C_n)$, where $n \ge 3$. Then
\begin{equation*}
DV(v) = \left\{
\begin{array}{lr}
\ 1 & \mbox{ if } n \equiv 0  \mbox{ (mod 3)} \ \\
\ \frac{1}{2} \lceil\frac{n}{3}\rceil (1+ \lceil\frac{n}{3}\rceil) & \mbox{ if } n \equiv 1  \mbox{ (mod 3)} \ \\
\ \lceil\frac{n}{3}\rceil & \mbox{ if } n \equiv 2 \mbox{ (mod 3)} .
\end{array} \right.
\end{equation*}
\end{Crlr}


\section{Total number of minimum dominating sets and domination value in $P_2 \square P_n$}

We consider $P_2 \square P_n$ ($n \ge 2$) as two copies
of $P_n$ with vertices labeled $x_1, x_2, \ldots, x_n$ and $y_1,
y_2, \ldots, y_n$ with only the edges $x_iy_i$, for each $i$ ($1
\le i \le n$), between two paths (see Figure \ref{2xn}).

\begin{figure}[htpb]
\begin{center}
\hspace*{1.43in}\scalebox{0.45}{\input{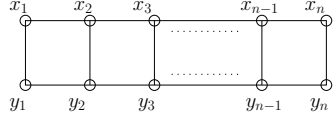}}
\vspace*{-2.3in}
\caption{Labeling of vertices of $P_2 \square P_n$}\label{2xn}
\end{center}
\end{figure}

We first recall the following.

\begin{Thm} \cite{Kinch}\label{domination 2n}
For $n \ge 2$, $\gamma(P_2 \square
P_n)=\lceil\frac{n+1}{2}\rceil$.
\end{Thm}

\begin{Lem}\label{corner}
Let $G=P_2 \square P_n$, where $n \ge 2$. If neither $x_1$ nor $y_1$ belongs to a $\gamma(G)$-set $D$, 
then $\{x_2, y_2\} \subseteq D$. (Likewise, if neither $x_n$ nor $y_n$ belongs to $D$, then $\{x_{n-1}, y_{n-1}\} \subseteq D$.)
\end{Lem}

\begin{proof}
By definition of a dominating set, either $x_1$ or a vertex in $N(x_1)=\{x_2, y_1\}$ belongs to $D$. 
If $x_1 \not\in D$ and $y_1 \not\in D$, then $x_2 \in D$. Similarly, either $y_1 \in D$ 
or a vertex in $N(y_1)=\{x_1, y_2\}$ belongs to $D$. If $x_1 \not\in D$ and $y_1 \not\in D$, 
then $y_2 \in D$ as well. Thus $x_1 \not\in D$ and $y_1 \not\in D$ implies $\{x_2, y_2\} \subseteq D$. \hfill
\end{proof}

\begin{Lem}\label{3,6}
Let $G=P_2 \square P_n$, where $n \ge 3$. If there exists a $\gamma(G)$-set containing no vertex of degree two, then $n=3$ or $n=6$.
\end{Lem}

\begin{proof}
Suppose that $D$ is a $\gamma(G)$-set such that $\{x_1, y_1, x_n, y_n\} \cap D=\emptyset$. Let $S_0=\{x_2, y_2, x_{n-1}, y_{n-1}\}$. Then, by Lemma \ref{corner}, 
$S_0 \subseteq D$. Note that $|S_0|=2$ if and only if $n=3$: in this case, $\gamma(P_2 \square P_3)=2$ and $S_0=\{x_2, y_2\}$ is a $\gamma(P_2 \square P_3)$-set.
If $4 \le n \le 5$, then $|S_0|=4$ and $\gamma(P_2 \square P_n)=3$, and thus $S_0 \not\subseteq D$. If $n=6$, then $|S_0|=4$ and $\gamma(P_2 \square P_6)=4$: in fact, $S_0=\{x_2, y_2, x_5, y_5\}$ is a $\gamma(P_2 \square P_6)$-set. Now, we need to consider $n \ge 7$. Suppose that $S_0 \subseteq D$; we consider two cases.

\textit{Case 1. $n=2k$, where $k \ge 4$:} Here, $\gamma(P_2 \square P_{2k})=k+1$. Since $N[S_0]=\{x_i, y_i \mid 1 \le i \le 3\} \cup \{x_j, y_j \mid 2k-2 \le j \le 2k\}$, the part of $P_2 \square P_{2k}$ not dominated by $S_0$ is a $P_2 \square P_{2k-6}$. So, $k-3$ vertices of $D - S_0$ must dominate $P_2 \square P_{2k-6}$. But $\gamma(P_2 \square P_{2k-6})=k-2$ by Theorem \ref{domination 2n}, and we reach a contradiction.

\textit{Case 2. $n=2k+1$, where $k \ge 3$:} Here, $\gamma(P_2 \square P_{2k+1})=k+1$. Since $N[S_0]=\{x_i, y_i \mid 1 \le i \le 3\} \cup \{x_j, y_j \mid 2k-1 \le j \le 2k+1\}$, the part of $P_2 \square P_{2k+1}$ not dominated by $S_0$ is a $P_2 \square P_{2k-5}$. So, $k-3$ vertices of $D - S_0$ must dominate $P_2 \square P_{2k-5}$. But $\gamma(P_2 \square P_{2k-5})=k-2$ by Theorem \ref{domination 2n}, and we reach a contradiction.

Thus, we have shown that if $S_0 \subseteq D$, then $n=3$ or $n=6$. \hfill
\end{proof}

Next we compute the total number of $\gamma(P_2 \square P_n)$-sets for $n \ge 2$.

\begin{Thm}\label{2n, odd}
For $n \ge 2$,
\begin{equation*} \tau(P_2 \square P_n)= \left\{
\begin{array}{ll}
6 & \mbox{ if } n=2\\
3 & \mbox{ if } n=3\\
17 & \mbox{ if } n=6\\
2 & \mbox{ if $n$ is odd and } n \neq 3\\
2n+4 & \mbox{ if $n$ is even and } n \neq 2,6 \ .
\end{array} \right.
\end{equation*}
\end{Thm}

\begin{proof}
Let $D$ be a $\gamma(P_2 \square P_n)$-set for $n \ge 2$. Notice that no $D$ contains both $x_1$ and $y_1$, or both $x_n$ and $y_n$, unless $n=2$. We consider two cases.

\textit{Case 1. $n \ge 3$ is odd:} Here, $\gamma(P_2 \square P_n)=\frac{n+1}{2}$. By Lemma \ref{3,6}, if there is a $D$ containing no vertex of degree two then $n=3$. Moreover, we note that $\{x_2, y_2\} \subseteq D$ if and only if $n=3$: If $\{x_2, y_2\} \subseteq  D$ and $n >3$, then the part of $P_2 \square P_{n}$ not dominated by $\{x_2, y_2\}$ is a $P_2 \square P_{n-3}$, and $\frac{n-3}{2}$ vertices of $D-\{x_2, y_2\}$ must dominate $P_2 \square P_{n-3}$. But $\gamma(P_2 \square P_{n-3})=\frac{n-1}{2}$ by Theorem \ref{domination 2n}, and we reach a contradiction. So, if $n>3$, by Lemma \ref{corner}, either $x_1 \in D$ or $y_1 \in D$. One can easily check that $x_1 \in D$ uniquely determines a $\gamma$-set $D=\{x_i, y_j \mid i \equiv 1, j \equiv 3 \pmod 4\}$. Similarly, $y_1 \in D$ uniquely determines a $\gamma$-set $D=\{x_i, y_j \mid i \equiv 3, j \equiv 1 \pmod 4\}$. Thus, $\tau(P_2 \square P_{n})=2$ for $n \neq 3$, and $\tau(P_2 \square P_3)=3$ by Lemma \ref{3,6}. (See Figure \ref{2x3} for the three $\gamma(P_2 \square P_3)$-sets, where the solid black vertices in each $P_2 \square P_3$ form a $\gamma(P_2 \square P_3)$-set.)
\begin{figure}[htpb]
\begin{center}
\scalebox{0.35}{\input{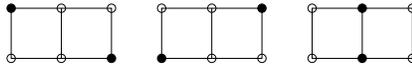}} \caption{$\gamma$-sets for $P_2 \square P_3$}\label{2x3}
\end{center}
\end{figure}

\textit{Case 2. $n \ge 2$ is even:} Here, $\gamma(P_2 \square P_n)=\frac{n}{2}+1$. If $n=2$, then $\gamma(P_2 \square P_2)=2$ and $\tau(P_2 \square P_2)=\tau(C_4)={4 \choose 2}=6$. We consider $n \ge 4$. 
By Lemma \ref{3,6}, if there is a $D$ containing no vertex of degree two (i.e., $\{x_2, y_2, x_{n-1}, y_{n-1}\} \subseteq D$), then $n=6$. We consider three subcases. 

\textit{Subcase 2.1. $\{x_2, y_2\} \subseteq D$ and $\{x_{n-1}, y_{n-1}\} \cap D = \emptyset$:} Let $\tau_1$ be the number of such $\gamma(P_2 \square P_{n})$-sets for $n \ge 4$. 
Note that the part of $P_2 \square P_{n}$ not dominated by $\{x_2, y_2\}$ is a $P_2 \square P_{n-3}$. So, $\tau_1$ equals the number of $\gamma(P_2 \square P_{n-3})$-sets with $\gamma(P_2 \square P_{n-3})=\frac{n}{2}-1$. One can easily see that $\tau_1=2$ when $n=4,6$. Since $\tau_1 (P_2 \square P_{n-3})=2$ for $n \ge 8$ by Case 1, we have $\tau_1=2$ for $n \ge 4$. 

\textit{Subcase 2.2. $\{x_2, y_2\} \cap D = \emptyset$ and $\{x_{n-1}, y_{n-1}\} \subseteq D$:} Let $\tau_2$ be the number of such $\gamma(P_2 \square P_{n})$-sets for $n \ge 4$. By Observation \ref{Obs2} and Subcase 2.1, we have $\tau_2=2$ for $n \ge 4$. 

\textit{Subcase 2.3. $\{x_2, y_2\} \not\subseteq D$ and $\{x_{n-1}, y_{n-1}\} \not\subseteq D$:} By Lemma \ref{corner}, $|\{x_1, y_1\} \cap D|=1$ and $|\{x_n, y_n\} \cap D|=1$. Let $D$ (resp. $D'$) be such a $\gamma$-set of $G=P_2 \square P_{n}$ (resp. $G'=P_2 \square P_{n+2}$), where $n \ge 4$. And let $\tau_3$ (resp. $\tau_3'$) be the number of such $\gamma$-sets of $G$ (resp. $G'$). We will show that $\tau_3=2n$, for $n \ge 4$, using induction. The base case, $n=4$, is easily verified (see Figure \ref{2x4}). Assume that $\tau_3=2n$ for $n \ge 4$.  
If $x_1 \in D$, then each $D$ extends to $D'$ such that $D'=D \cup \{x_{n+2}\}$ if $y_{n} \in D$ and $D'=D \cup \{y_{n+2}\}$ if $x_{n} \in D$; in addition, there are two additional $\gamma(G')$-sets which do not come from any $\gamma(G)$-sets, i.e., $\{x_i, y_j \mid i \equiv 1, j \equiv 3 \pmod 4 \mbox{ and } 1 \le i,j \le n+1\} \cup \{x_{n+2}\}$ and $\{x_i, y_j \mid i \equiv 1, j \equiv 3 \pmod 4 \mbox{ and } 1 \le i,j \le n+1\} \cup \{y_{n+2}\}$. Similarly, if $y_1 \in D$, then each $D$ extends to $D'$ and there are two additional $\gamma(G')$-sets which do not come from $\gamma(G)$-sets. So, $\tau_3'=\tau_3+4=2n+4=2(n+2)$.

Now, noting that $\{x_2, y_2, x_{n-1}, y_{n-1}\} \subseteq D$ implies $n=6$, combine the three disjoint cases to get $\tau=\tau_1+\tau_2+\tau_3=2+2+2n=2n+4$ if $n \neq 2, 6$ and $\tau(P_2 \square P_6)=(2 \cdot 6+4)+1=17$. \hfill
\end{proof}

See Figure \ref{2x4} for the collection of $\gamma(P_2 \square P_4)$-sets, where the solid black vertices in each $P_2 \square P_4$ form a $\gamma(P_2 \square P_4)$-set.

\begin{figure}[htpb]
\begin{center}
\scalebox{0.43}{\input{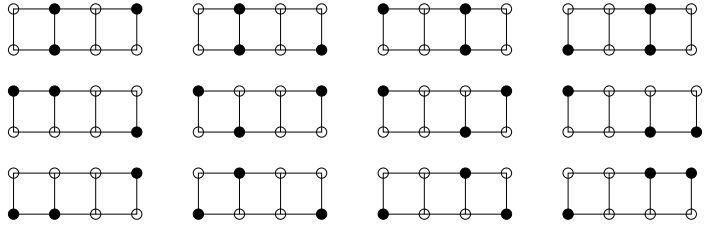}} \caption{$\gamma$-sets for $P_2 \square P_4$}\label{2x4}
\end{center}
\end{figure}

As an immediate consequence of Theorem \ref{2n, odd} for an odd $n
\ge 3$, we have the following

\begin{Crlr}\label{2nodd}
Let $n \ge 3$ be an odd number.
\begin{itemize}
\item[(i)] For each $v \in V(P_2 \square P_3)$, $DV(v)=1$.
\item[(ii)] For $x_i, y_i \in V(P_2 \square P_n)$, where $n \ge
5$,
\begin{equation*} DV(x_i)=DV(y_i)= \left\{
\begin{array}{ll}
1 & \mbox{ if $i$ is odd} \\
0 & \mbox{ if $i$ is even} \ .
\end{array} \right.
\end{equation*}
\end{itemize}
\end{Crlr}

\begin{Prop}
Let $n \ge 2$ be an even number.
\begin{itemize}
\item[(i)] For each $v \in V(P_2 \square P_2)$, $DV(v)=3$.
\item[(ii)] For $x_i, y_i \in V(P_2 \square P_n)$, where $n \ge
4$ and $n \neq 6$,
\begin{equation} \label{2xnevenDV}
\hspace{-0.18in}DV(x_i)=DV(y_i)=\left\{
\begin{array}{ll}
n+2-i& \mbox{ if $i$ is odd and }1 \le i \le n-3\\
4& \mbox{ if $i=2$ or }i=n-1\\
i+1& \mbox{ if $i$ is even and }4 \le i \le n \ .
\end{array} \right.
\end{equation}
\item[(iii)] For $x_i, y_i \in V(P_2 \square
P_6)$,
\begin{equation}\label{DV2x6}
DV(x_i)=DV(y_i)= \left\{
\begin{array}{ll}
7 & \mbox{ if $i=1$ or $i=6$} \\
5 & \mbox{ if $2 \le i \le 5$} \ .
\end{array} \right.
\end{equation}
\end{itemize}
\end{Prop}

\begin{proof}
Let $n \ge 2$ be an even number.

(i) Note that $P_2 \square P_2 \cong C_4$, $\gamma(C_4)=2$, and $\tau(C_4)=6$. By Observation
\ref{Obs1}, Observation \ref{Obs2}, and the vertex-transitivity, $DV(v)=3$ for each $v
\in V(P_2 \square P_2)$.

\vspace{.05in}

(ii) For an even $n \ge 4$, let $D$ (resp. $D'$) be a $\gamma$-set of $G=P_2 \square P_n$ (resp. $G'=P_2 \square P_{n+2}$). Since $DV_G(x_i)=DV_G(y_i)$ for each $i$ ($1 \le i \le n$), it suffices to compute $DV_G(x_i)$ for $1 \le i \le n$. We consider two cases.

\textit{Case 1. $\{x_1, y_1\} \cap D= \emptyset$:} By Lemma \ref{corner}, $\{x_2, y_2\}  \subseteq D$. Denote by $DV^1(v)$ the number of such $D$'s containing $v$. Notice that there are two such $\gamma(G)$-sets. We will show, by induction, that 
\begin{equation}\label{dv1}
DV_{G}^1(x_i)= \left\{
\begin{array}{ll}
2 & \mbox{ if $i=2$} \\
1 & \mbox{ if $i \ge 4$ and $i$ is even} \\
0 & \mbox{ if $i$ is odd} \ .
\end{array} \right.
\end{equation}
For $n=4$ (the base case), the two $\gamma$-sets are 
$\{x_2, y_2, x_4\}$ and $\{x_2, y_2, y_4\}$, thus satisfying (\ref{dv1}).  Assume that (\ref{dv1}) holds for $G$. Let $D_1$ and $D_2$ be $\gamma(G)$-sets, containing both $x_2$ and $y_2$, such that $x_n \in D_1$ and $y_n \in D_2$. Then $D_1$ extends to $D_1'=D_1 \cup \{y_{n+2}\}$ and $D_2$ extends to $D_2'=D_2 \cup \{x_{n+2}\}$, where $D_1'$ and $D_2'$ are $\gamma(G')$-sets. So, $DV_{G'}^1(x_i)=DV_G^1(x_i)$ for $1 \le i \le n$, $DV_{G'}^1(x_{n+1})=0$, and $DV_{G'}^1(x_{n+2})=1$. Thus
\begin{equation*}
DV_{G'}^1(x_i)=\left\{
\begin{array}{ll}
2 & \mbox{ if $i=2$} \\
1 & \mbox{ if $i \ge 4$ and $i$ is even} \\
0 & \mbox{ if $i$ is odd} \ ,
\end{array} \right.
\end{equation*}
proving (\ref{dv1}).

\textit{Case 2. $x_1 \in D$ or $y_1 \in D$:} Denote by $DV^2(v)$ the number of such $D$'s containing $v$. By Subcase 2.2 and Subcase 2.3 in the proof of Theorem \ref{2n, odd}, there are $2n+2$ such $\gamma(G)$-sets; $n+1$ such $D$'s containing $x_1$, and $n+1$ such $D$'s containing $y_1$. We will show, by induction, that
\begin{equation}\label{dv0}
DV_{G}^2(x_{i})=\left\{
\begin{array}{ll}
i & \mbox{ if } i \equiv 0,2 \pmod 4 \mbox{ and } 2 \le i \le n\\
n+2-i & \mbox{ if } i \equiv 1,3 \pmod 4 \mbox{ and } 1 \le i \le n-3\\
4 & \mbox{ if } i=n-1 \ .
\end{array} \right.
\end{equation} 

Noting that no $\gamma(G)$-set contains both $x_1$ and $y_1$, we consider two subcases.

\textit{Subcase 2.1. $x_1 \in D$:} Denote by $DV^{2,1}(v)$ the number of such $D$'s containing $v$. For $n=4$ (the base case), one can check that there are five such $\gamma$-sets: $\{x_1, x_2, y_4\}$, $\{x_1, y_2, x_4\}$, $\{x_1, y_3, x_4\}$, $\{x_1, y_3, y_4\}$,  and $\{x_1, x_3, y_3\}$. Let $D_1, D_2, \cdots, D_{n+1}$ be $\gamma(G)$-sets containing $x_1$, where $\{x_{n-1}, y_{n-1}\} \subseteq D_{n+1}$. Then, for $1 \le i \le n$, each $D_i$ extends to $D_i'=D_i \cup \{x_{n+2}\}$ if $y_n \in D_i$ and  $D_i'=D_i \cup \{y_{n+2}\}$ if $x_n \in D_i$, where each $D_i'$ ($1 \le i \le n$) is a $\gamma(G')$-set; $D_{n+1}=\{x_i, y_j \mid i \equiv 1, j \equiv 3 \pmod 4 \mbox{ and } 1 \le i,j \le n-2\} \cup \{x_{n-1}, y_{n-1}\}$ does not extend to a $\gamma(G')$-set, but there exists a $\gamma(G')$-set $D_{n+1}'=\{x_i, y_j \mid i \equiv 1, j \equiv 3 \pmod 4 \mbox{ and } 1 \le i, j \le n\} \cup \{x_{n+1}, y_{n+1}\}$ which does not come from any $\gamma(G)$-set. Further, there exist two additional $\gamma(G')$-sets which do not come from any $\gamma(G)$-sets such as $D_{n+2}'=\{x_i, y_j \mid i \equiv 1, j \equiv 3 \pmod 4 \mbox{ and } 1 \le i, j \le n+1\} \cup \{x_{n+2}\}$ and $D_{n+3}'=\{x_i, y_j \mid i \equiv 1, j \equiv 3 \pmod 4 \mbox{ and } 1 \le i, j \le n+1\} \cup \{y_{n+2}\}$. So, noting that $n$ is even, we have the following:
\begin{equation*}
DV_{G'}^{2,1}(x_i)= \left\{
\begin{array}{ll}
\vspace{.05in}
DV_G^{2,1}(x_i) & \mbox{ if } i \equiv 0,2,3 \pmod 4 \mbox{ and } 1 \le i \le n-2 \\
DV_G^{2,1}(x_i)+2 & \mbox{ if } i \equiv 1 \pmod 4 \mbox{ and } 1 \le i \le n-2 \ ,
\end{array} \right.
\end{equation*}
\begin{equation*}
DV_{G'}^{2,1}(x_{n-1})= \left\{
\begin{array}{ll}
\vspace{.05in}
DV_G^{2,1}(x_{n-1})-1 & \mbox{ if } n \equiv 0 \pmod 4 \\
DV_G^{2,1}(x_{n-1})+2 & \mbox{ if } n \equiv 2 \pmod 4 \ ,
\end{array} \right.
\end{equation*} 
\begin{equation*}
DV_{G'}^{2,1}(x_{n+1})=\left\{
\begin{array}{ll}
3 & \mbox{ if } n \equiv 0 \pmod 4 \\
1 & \mbox{ if } n \equiv 2 \pmod 4 \ ,
\end{array} \right.
\end{equation*}
$DV_{G'}^{2,1}(x_n)=DV_{G}^{2,1}(x_n)$, and $DV_{G'}^{2,1}(x_{n+2})=\frac{n}{2}+1$.

\textit{Subcase 2.2. $y_1 \in D$:} Denote by $DV^{2,2}(v)$ the number of such $D$'s containing $v$. For $n=4$ (the base case), one can check that there are five such $\gamma$-sets: $\{y_1, y_2, x_4\}$, $\{y_1, x_2, y_4\}$, $\{y_1, x_3, x_4\}$, $\{y_1, x_3, y_4\}$,  and $\{y_1, x_3, y_3\}$. Let $\Gamma_1, \Gamma_2, \cdots, \Gamma_{n+1}$ be $\gamma(G)$-sets containing $y_1$, where $\{x_{n-1}, y_{n-1}\} \subseteq \Gamma_{n+1}$. Then, for $1 \le i \le n$, each $\Gamma_i$ extends to $\Gamma_i'=\Gamma_i \cup \{x_{n+2}\}$ if $y_n \in \Gamma_i$ and $\Gamma_i'=\Gamma_i \cup \{y_{n+2}\}$ if $x_n \in \Gamma_i$, where each $\Gamma_i'$ ($1 \le i \le n$) is a $\gamma(G')$-set; $\Gamma_{n+1}=\{x_i, y_j \mid i \equiv 3, j \equiv 1 \pmod 4 \mbox{ and } 1 \le i, j \le n-2\} \cup \{x_{n-1}, y_{n-1}\}$ does not extend to a $\gamma(G')$-set, but there exists a $\gamma(G')$-set $\Gamma_{n+1}'=\{x_i, y_j \mid i \equiv 3, j \equiv 1 \pmod 4 \mbox{ and } 1 \le i, j \le n\} \cup \{x_{n+1}, y_{n+1}\}$ which does not come from any $\gamma(G)$-set. Further, there exist two additional $\gamma(G')$-sets which do not come from any $\gamma(G)$-sets such as $\Gamma_{n+2}'=\{x_i, y_j \mid i \equiv 3, j \equiv 1 \pmod 4 \mbox{ and } 1 \le i, j \le n+1\} \cup \{x_{n+2}\}$ and $\Gamma_{n+3}'=\{x_i, y_j \mid i \equiv 3, j \equiv 1 \pmod 4 \mbox{ and } 1 \le i, j \le n+1\} \cup \{y_{n+2}\}$. So, noting that $n$ is even, we have the following:
\begin{equation*}
DV_{G'}^{2,2}(x_i)= \left\{
\begin{array}{ll}
\vspace{.05in}
DV_G^{2,2}(x_i) & \mbox{ if } i \equiv 0,1,2 \pmod 4 \mbox{ and } 1 \le i \le n-2 \\
DV_G^{2,2}(x_i)+2 & \mbox{ if } i \equiv 3 \pmod 4 \mbox{ and } 1 \le i \le n-2 \ ,
\end{array} \right.
\end{equation*}
\begin{equation*}
DV_{G'}^{2,2}(x_{n-1})= \left\{
\begin{array}{ll}
\vspace{.05in}
DV_G^{2,2}(x_{n-1})+2 & \mbox{ if } n \equiv 0 \pmod 4 \\
DV_G^{2,2}(x_{n-1})-1 & \mbox{ if } n \equiv 2 \pmod 4 \ ,
\end{array} \right.
\end{equation*}  
\begin{equation*}
DV_{G'}^{2,2}(x_{n+1})=\left\{
\begin{array}{ll}
1 & \mbox{ if } n \equiv 0 \pmod 4 \\
3 & \mbox{ if } n \equiv 2 \pmod 4 \ ,
\end{array} \right.
\end{equation*}
$DV_{G'}^{2,2}(x_n)=DV_{G}^{2,2}(x_n)$, and $DV_{G'}^{2,2}(x_{n+2})=\frac{n}{2}+1$.

Next, assume that (\ref{dv0}) holds for $G$. Noting that $DV^2(v)=DV^{2,1}(v)+DV^{2,2}(v)$ and that $n$ is even, by Subcase 2.1 and Subcase 2.2, we have
\begin{equation*}
DV_{G'}^2(x_i)= \left\{
\begin{array}{ll}
\vspace{.05in}
DV_G^2(x_i) & \mbox{ if } i \equiv 0,2 \pmod 4 \mbox{ and } 1 \le i \le n-2 \\
DV_G^2(x_i)+2 & \mbox{ if } i \equiv 1,3 \pmod 4 \mbox{ and } 1 \le i \le n-2 \ ,
\end{array} \right.
\end{equation*}
$DV_{G'}^2(x_{n-1})=DV_G^2(x_{n-1})+1$, $DV_{G'}^2(x_n)=DV_{G}^2(x_n)$, $DV_{G'}^2(x_{n+1})=4$, and $DV_{G'}^2(x_{n+2})=n+2$, proving (\ref{dv0}).

Now, noting that $DV(v)=DV^1(v)+DV^2(v)$ for $v \in V(P_2 \square P_n)$, where $n \ge 4$ is even and $n \neq 6$, combine (\ref{dv1}) and (\ref{dv0}) to obtain (\ref{2xnevenDV}), proving $(ii)$.

\vspace{.05in}

(iii) By Theorem \ref{2n, odd}, $P_2 \square P_6$ has an additional $\gamma$-set $\{x_2, y_2, x_5, y_5\}$. This,
together with (\ref{2xnevenDV}), for $x_i, y_i \in V(P_2 \square P_6)$, we
obtain
\begin{equation*}DV(x_i)=DV(y_i)=
\left\{
\begin{array}{ll}
8-i & \mbox{ if $i$ is odd and $1 \le i \le 3$}\\
5 & \mbox{ if $i=2$ or $i=5$}\\
i+1 & \mbox{ if $i$ is even and $4 \le i \le 6$ },
\end{array} \right.
\end{equation*}
which equals the domination value in (\ref{DV2x6}). \hfill
\end{proof}


\section{Total number of minimum dominating sets and domination value in $P_2 \square C_n$}

For $n \ge 3$, consider $P_2 \square C_n$ as two copies
of $C_n$ with vertices labeled $x_1, x_2, \ldots, x_n$ and $y_1,
y_2, \ldots, y_n$ with only the edges $x_iy_i$, for each $i$ ($1
\le i \le n$), between two cycles (see Figure \ref{2cn}).
\begin{figure}[htpb]
\begin{center}
\hspace*{1.2in}\scalebox{0.45}{\input{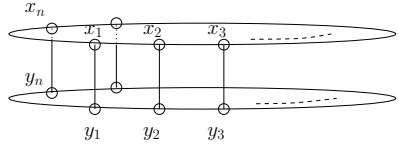}}
\vspace*{-2.24in}
\caption{Labeling of vertices of $P_2 \square C_n$}\label{2cn}
\end{center}
\end{figure}

\vspace{-.2in}
We recall the following result.

\begin{Thm} \cite{FD}
For $n \ge 3$,
\begin{equation*}
\gamma(P_2 \square C_n)=
\left\{
\begin{array}{ll}
\vspace{.05in}
\frac{n}{2} & \mbox{ if } n \equiv 0 \pmod 4\\
\lceil \frac{n+1}{2} \rceil & \mbox{ if } n \not\equiv 0 \pmod 4 .
\end{array} \right.
\end{equation*}
\end{Thm}

We introduce the following definition which will be used in the proof of Theorem \ref{2Cn}.

\begin{Def}
Let $G^1$ and $G^2$ be disjoint copies of a graph $G$, and let $D$ be a $\gamma(P_2 \square G)$-set. Let $\langle D \cap V(G^1) \rangle=\cup_{i=1}^{m_1} \mathcal{H}_i^1$, a disjoint union of connected components such that $|V(H_i^1)| \le |V(H_{i+1}^1)|$ for $1 \le i \le m_1-1$; similarly, we write $\langle D \cap V(G^2) \rangle=\cup_{i=1}^{m_2} \mathcal{H}_i^2$. Let $\alpha=\max(|V(\mathcal{H}_{m_1}^1)|, |V(\mathcal{H}_{m_2}^2)|)$; we will denote by $\mathcal{H}_{\alpha}$ any $\mathcal{H}_i^j$ with $|V(\mathcal{H}_i^j)|=\alpha$, for $j=1,2$ ($1 \le i \le m_1$ or $1 \le i \le m_2$). 
\end{Def}

\textbf{Example.} The black vertices in Figure \ref{def} form a $\gamma(P_2 \square C_{10})$-set $D$, where $\langle D \rangle$ contains $2\mathcal{H}_2$. 
\begin{figure}[htpb]
\begin{center}
\hspace*{.5in}
\scalebox{0.4}{\input{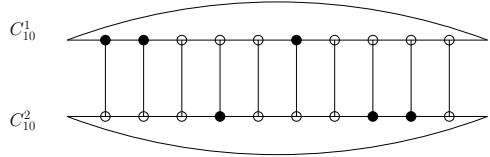}}
\vspace*{-.08in}
\caption{$2\mathcal{H}_2 \subseteq \langle D \rangle$, where $D$ is a $\gamma(P_2 \square C_{10})$-set}\label{def}
\end{center}
\end{figure}

\begin{Thm}\label{2Cn}
Let $n \ge 3$. For each $v \in  V(P_2 \square C_n)$, 
\begin{equation*}
DV(v)=
\left\{
\begin{array}{ll}
\vspace{.05in}
1 & \mbox{ if } n \equiv 0 \pmod 4\\
\vspace{.05in}
\frac{n+1}{2}  & \mbox{ if } n \equiv 1, 3 \pmod 4 \mbox{ and } n \neq 3\\
\vspace{.05in}
(\lceil \frac{n+1}{2} \rceil)^2 & \mbox{ if } n \equiv 2 \pmod 4 \mbox{ and } n \neq 6\\
\vspace{.05in}
3 & \mbox{ if } n=3\\
\vspace{.05in}
17 & \mbox{ if } n=6 \ .
\end{array} \right.
\end{equation*}
\end{Thm}

\begin{proof}
By Observation \ref{Obs2} and the vertex-transitivity, $DV(v)=DV(x_1)$ for each $v \in V(P_2 \square C_n)$. Let $D$ be a $\gamma(P_2 \square C_n)$-set containing $x_1$, where $n \ge 3$; note that at least a vertex in $\{x_2, x_3, y_1, y_2, y_3\}$ belongs to $D$. Noting that each vertex dominates four vertices, we consider four cases.

\textit{Case 1. $n=4k$, where $k \ge 1$:} Since $\gamma(P_2 \square C_{4k})=2k$ and $|V(P_2 \square C_{4k})|=8k$, each vertex is dominated by exactly one vertex (i.e., no vertex is doubly dominated). Thus there is a unique $D$ containing $x_1$, i.e., $D=\{x_i, y_j \mid i \equiv 1, j \equiv 3 \pmod 4\}$, and hence $DV(x_1)=1$.

\vspace{.07in}

\textit{Case 2. $n=4k+1$, where $k \ge 1$:} Here $\gamma(P_2 \square C_{4k+1})=2k+1$. We will show that no $D $ contains both $x_1$ and a vertex in $\{y_1, y_2, x_3\}$. First, we note that no $D$ contains both $x_1$ and $y_1$: if $\{x_1, y_1\} \subseteq D$, then the part of $P_2 \square C_{4k+1}$ not dominated by $\{x_1, y_1\}$ is a $P_2 \square P_{4k-2}$, and $2k-1$ vertices of $D-\{x_1, y_1\}$ must dominate $P_2 \square P_{4k-2}$. But $\gamma(P_2 \square P_{4k-2})=2k$ by Theorem \ref{domination 2n}, and we reach a contradiction. Second, we note that no $D$ contains both $x_1$ and  $y_2$: if $\{x_1, y_2\} \subseteq D$, then the part of $P_2 \square C_{4k+1}$ not dominated by $\{x_1, y_2\}$ is the graph $H$ in Figure \ref{x1y2}, and $2k-1$ vertices of $D-\{x_1, y_2\}$ must dominate $H$. 
\begin{figure}[htpb]
\begin{center}
\hspace*{1in}
\scalebox{0.45}{\input{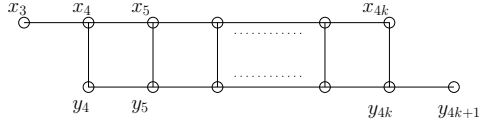}}
\vspace*{-2.35in}
\caption{$H \subset P_2 \square C_{4k+1}$}\label{x1y2}
\end{center}
\end{figure}
If we let $S_0=\{x_i, y_j \mid i \equiv 0, j \equiv 2 \pmod 4 \mbox{ and } 4 \le i, j \le 4k-2\}$, then $|S_0|=2(k-1)$, $S_0$ dominates $8(k-1)$ vertices, the part of $H$ not dominated by $S_0$ is a $P_4$, and one vertex of $D-(S_0 \cup \{x_1, y_2\})$ must dominate $P_4$. But $\gamma(P_4)=2$, and we reach a contradiction. (Similarly, no $D$ contains both $x_1$ and $y_{4k+1}$.) Third, no $D$ contains both $x_1$ and $x_3$: if $\{x_1, x_3\} \subseteq D$, then a vertex in $N[y_2]=\{x_2, y_1, y_2, y_3\}$ must belong to $D$. Since $\{x_1, y_1\} \not\subseteq D$ (and thus $\{x_3, y_3\} \not\subseteq D$ by the vertex-transitivity) and $\{x_1, y_2\} \not\subseteq D$, $x_2 \in D$. If $R_0:=\{x_1, x_2, x_3\} \subseteq D$, then the part of $P_2 \square C_{4k+1}$ not dominated by $R_0$, say $H_1$, must be dominated by $2k-2$ vertices in $D-R_0$. Since $|V(P_2 \square C_{4k+1})|=8k+2$ and $|N[R_0]|=8$, $2k-2$ vertices in $D-R_0$ must dominate $8k-6$ vertices. But each vertex in $P_2 \square C_{4k+1}$ dominates four vertices, and we reach a contradiction. (Similarly, no $D$ contains both $x_1$ and $x_{4k}$.)  So, we only need to consider $D$ such that (i) $\{x_1, x_2\} \subseteq D$ (resp. $\{x_1, x_{4k+1}\} \subseteq D$) or (ii) no vertex in $N[x_1]$ is doubly dominated (i.e., $\{x_1, y_3\} \subseteq D$ and $\{x_1, y_{4k}\}\subseteq D$). 

\textit{Subcase 2.1. $\{x_1, x_2\} \subseteq D$ (resp. $\{x_1, x_{4k+1}\} \subseteq D$):} The part of $P_2 \square C_{4k+1}$ not dominated by $\{x_1, x_2\}$, say $H_2$, must be dominated by $2k-1$ vertices in $D-\{x_1, x_2\}$. Since $|V(P_2 \square C_{4k+1})|=8k+2$ and $|N[\{x_1, x_2\}]|=6$, $2k-1$ vertices in $D-\{x_1,x_2\}$ must dominate $H_2$ with $|V(H_2)|=8k-4$, and thus there exists at most one $\gamma$-set containing both $x_1$ and $x_2$ (resp. $x_1$ and $x_{4k+1}$). Noting that $\{x_1\} \cup \{x_i, y_j \mid i \equiv 2, j \equiv 0 \pmod 4\}$ (resp. $\{x_i, y_j \mid i \equiv 1, j \equiv 3 \pmod 4\}$) is a $\gamma$-set, there is a unique $D$ containing both $x_1$ and $x_2$ (resp. $x_1$ and $x_{4k+1}$).

\textit{Subcase 2.2. No vertex in $N[x_1]$ is doubly dominated:} Since $x_1 \not\in V(\mathcal{H}_2)$, by Subcase 2.1, there are $2k-1$ slots in which $\mathcal{H}_2$ can be placed.

By Subcase 2.1 and Subcase 2.2, we have $DV(x_1)=2(1)+(2k-1)=2k+1$.

\vspace{.07in}

\textit{Case 3. $n=4k+2$, where $k \ge 1$:} Here $\gamma(P_2 \square C_{4k+2})=2k+2$. We will show that no $D$ contains a $\mathcal{H}_{\alpha}$ for $\alpha \ge 4$. If $R_1:=\{x_1, x_2, x_3, x_4\} \subseteq D$, then the part of $P_2 \square C_{4k+2}$ not dominated by $R_1$, say $F_1$, must be dominated by $2k-2$ vertices in $D-R_1$. Since $|V(P_2 \square C_{4k+2})|=8k+4$ and $|N[R_1]|=10$, $2k-2$ vertices in $D-R_1$ must dominate $F_1$ with $|V(F_1)|=8k-6$. But each vertex in $P_2 \square C_{4k+2}$ dominates four vertices, and we reach a contradiction. We consider four subcases. 

\textit{Subcase 3.1. $\mathcal{H}_3 \subseteq \langle D \rangle$:} We denote by $DV^1(x_1)$ the number of such $D$'s containing $x_1$. We note that the placement of $\mathcal{H}_3$ uniquely determines $D$: if $R_2:=\{x_1, x_2, x_3\} \subseteq D$, then the part of $P_2 \square C_{4k+2}$ not dominated by $R_2$, say $F_2$, must be dominated by $2k-1$ vertices in $D-R_2$. Since $|V(P_2 \square C_{4k+2})|=8k+4$ and $|N[R_2]|=8$, $2k-1$ vertices in $D-R_2$ must dominate $F_2$ with $|V(F_2)|=8k-4$, and thus there exists at most one $\gamma$-set containing $R_2$. Noting that $\{x_1, x_2\} \cup \{x_i, y_j \mid i \equiv 3, j\equiv 1 \pmod 4 \mbox{ and } 3 \le i,j \le 4k+2\}$ is a $\gamma$-set, there is a unique $D$ containing $R_2$. If $x_1 \in V(\mathcal{H}_3)$, there are three such $D$'s, i.e., $\{x_1, x_2, x_3\} \subseteq D$, $\{x_{4k+2}, x_1, x_2\} \subseteq D$, and $\{x_{4k+1}, x_{4k+2}, x_1\} \subseteq D$. If $x_1 \not\in V(\mathcal{H}_3)$, there are $2k-1$ slots in which $\mathcal{H}_3$ can be placed. So, $DV^1(x_1)=3+(2k-1)=2k+2$.

\textit{Subcase 3.2. $2\mathcal{H}_2 \subseteq \langle D \rangle$:} We denote by $DV^2(x_1)$ the number of such $D$'s containing $x_1$. Since each vertex in $\mathcal{H}_2$ is doubly dominated, four vertices in $2\mathcal{H}_2$ are doubly dominated, and hence the placement of $2 \mathcal{H}_2$ uniquely determines $D$. If $x_1 \in V(\mathcal{H}_2)$ (i.e., $\{x_1, x_2\} \subseteq D$ or $\{x_1, x_{4k+2}\} \subseteq D$), then there are $2k-1$ available slots to place the other $\mathcal{H}_2$. If $x_1 \not\in V(\mathcal{H}_2)$, then there are ${2k-1 \choose 2}$ available slots to place $2\mathcal{H}_2$'s. Thus, $DV^2(x_1)$ $=2(2k-1)+{2k-1 \choose 2}=(2k-1)(k+1)$. 

\textit{Subcase 3.3. $\mathcal{H}_2 \subseteq \langle D \rangle$ and $2\mathcal{H}_2 \not\subseteq \langle D \rangle$:} We will show that no such $D$ exists. Without loss of generality, suppose that $\{x_1, x_2\} \subseteq D$. In order for $y_3$ to be dominated, a vertex in $N[y_3]=\{x_3, y_2, y_3, y_4\}$ must be in $D$. By the hypothesis, $\{x_1, x_2, x_3\} \not\subseteq D$. First, suppose that $R_3:=\{x_1, x_2, y_2\} \subseteq D$. Then the part of $P_2 \square C_{4k+2}$ not dominated by $R_3$, say $F_3$, must be dominated by $2k-1$ vertices in $D-R_3$. Since $|V(P_2 \square C_{4k+2})|=8k+4$ and $|N[R_3]|=7$, $2k-1$ vertices in $D-R_3$ must dominate $F_3$ with $|V(F_3)|=8k-3$. But each vertex in $P_2 \square C_{4k+2}$ dominates four vertices, and we reach a contradiction. Second, suppose that $R_4:=\{x_1, x_2, y_3\} \subseteq D$. Then the part of $P_2 \square C_{4k+2}$ not dominated by $R_4$, say $F_4$, is a graph isomorphic to $H$ in Figure \ref{x1y2}, and $2k-1$ vertices of $D-R_4$ must dominate $F_4 \cong H$, which is a contradiction by Case 2. Third, suppose that $R_5:=\{x_1,x_2,y_4\} \subseteq D$. Then the part of $P_2 \square C_{4k+2}$ not dominated by $R_5$, say $F_5$, must be dominated by $2k-1$ vertices in $D-R_5$. Since $|V(P_2 \square C_{4k+2})|=8k+4$ and $|N[R_5]|=10$, $2k-1$ vertices in $D-R_5$ must dominate $F_5$ with $|V(F_5)|=8k-6$, and thus there exist two vertices in $N[F_5]$ that are doubly dominated. When $k=1$, one can easily see that $y_5 \in D$ (i.e., $2\mathcal{H}_2 \subseteq \langle D \rangle$) or $x_6 \in D$ (i.e., $\mathcal{H}_3 \subseteq \langle D \rangle$); both cases contradict to the assumption. So we consider for $k \ge 2$. Without loss of generality, we may assume that at least one vertex in $N[y_4] \cap N[F_5]=\{x_4, y_5\}$ is doubly dominated. In order for $x_4$ to be doubly dominated, $x_5 \in D$. If $\{x_1, x_2, y_4, x_5\} \subseteq D$, then the part of $P_2 \square C_{4k+2}$ not dominated by $\{x_1, x_2, y_4, x_5\}$ is the graph $H'$ in Figure \ref{C4k+2}, 
\begin{figure}[htpb]
\begin{center}
\hspace*{1in}\scalebox{0.45}{\input{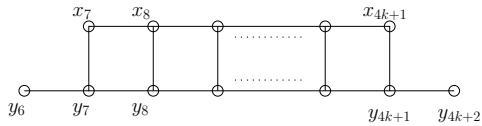}}
\vspace*{-2.3in}
\caption{$H' \subset P_2 \square C_{4k+2}$, where $k \ge 2$}\label{C4k+2}
\end{center}
\end{figure}and $2k-2$ vertices of $D-\{x_1, x_2, y_4, x_5\}$ must dominate $H'$. If we let $S'=\{x_i, y_j \mid i \equiv 1, j \equiv 3 \pmod 4 \mbox{ and } 6 \le i, j \le 4k\}$, then $|S'|=2k-3$, $S'$ dominates $8k-12$ vertices, the part of $H'$ not dominated by $S'$ is a $P_4$, and one vertex of $D-(S' \cup \{x_1, x_2, y_4, x_5\})$ must dominate $P_4$. But $\gamma(P_4)=2$ and we reach  a contradiction. In order for $y_5$ to be doubly dominated, a vertex in $\{x_5, y_5, y_6\}$ must belong to $D$. Since $\{x_1, x_2, y_4, x_5\} \not\subseteq D$ and $\{x_1, x_2, y_4, y_5\} \not\subseteq D$, $y_6 \in D$. In this case, i.e., $\{x_1, x_2, y_4, y_6\} \subseteq D$, note that $x_1$, $x_2$, and $y_5$ are doubly dominated. In order for $x_5$ to be dominated, a vertex in $N[x_5]=\{x_4, x_5, x_6, y_5\}$ must be in $D$ and each case results in at least two additional vertices to be doubly dominated, which is a contradiction. Thus, there is no $\gamma(P_2 \square C_{4k+2})$-set containing exactly one $\mathcal{H}_2$.

\textit{Subcase 3.4. $\mathcal{H}_2 \not\subseteq \langle D \rangle$:} We denote by $DV^3(x_1)$ the number of such $D$'s containing $x_1$. First, suppose that $\{x_s, y_s\} \subseteq D$ for some $s$ ($1 \le s \le 4k+2$). If $\{x_1, y_1\} \subseteq D$, then the part of $P_2 \square C_{4k+2}$ not dominated by $\{x_1, y_1\}$ is $P_2 \square P_{4k-1}$, and $2k$ vertices of $D-\{x_1, y_1\}$ must dominate $P_2 \square P_{4k-1}$. By Theorem \ref{2n, odd}, there exist two such $D$'s for $k \neq 1$ (i.e., $n \neq 6$) and there exist three such $D$'s for $k=1$ (i.e., $n=6$). If $x_1 \in D$ and $\{y_1, y_2, y_{4k+2} \} \cap D = \emptyset$, then there are $2k$ available slots in which $\{x_s, y_s\} \subseteq D$ can be placed for some $s \neq 1$. Second, suppose that no two adjacent vertices belong to $D$. If we let $S_1=\{x_i, y_j \mid i \equiv 1, j \equiv 3 \pmod 4 \mbox{ and } 1 \le i,j \le 4k\}$, then $|S_1|=2k$ and the part of $P_2 \square C_{4k+2}$ not dominated by $S_1$ is a $P_4$, so two vertices of  $D-S_1$ must dominate $P_4$. Since no two adjacent vertices belong to $D$, if $S_1 \subseteq D$, then $\{x_{4k}, y_{4k+1}\} \subseteq D$ or $\{x_{4k}, y_{4k+2}\} \subseteq D$ or $\{x_{4k+1}, y_{4k+2}\} \subseteq D$, thus there are two pairs of vertices (not necessarily disjoint) in $D$ that are at distance two apart. The number of ways of selecting 2 out of $2k+2$ available slots is ${2k+2 \choose 2}=(k+1)(2k+1)$. Thus, $DV^3(x_1)=2+2k+(k+1)(2k+1)=(k+1)(2k+3)$ if $k \neq 1$, and $DV^3(x_1)=11$ if $k=1$.

Now, noting that $DV(x_1)=DV^1(x_1)+DV^2(x_1)+DV^3(x_1)$, we have $DV(x_1)=(2k+2)^2$ if $k \neq 1$, and $DV(x_1)=17$ if $k=1$.

\vspace{.07in}

\textit{Case 4. $n=4k+3$, where $k \ge 0$:} Here $\gamma(P_2 \square C_{4k+3})=2k+2$. When $k=0$, one can easily check that there are three $\gamma$-sets containing $x_1$, i.e., $\{x_1, y_1\}$, $\{x_1, y_2\}$, and $\{x_1, y_3\}$. So $DV(x_1)=3$ for $x_1 \in V(P_2 \square C_3)$. Next, we consider for $k \ge 1$. We will show that no $D$ contains both $x_1$ and a vertex in $\{y_1, x_2, x_3\}$. First, note that no $D$ contains both $x_1$ and $y_1$: If $\{x_1, y_1\} \subseteq D$, then the part of $P_2 \square C_{4k+3}$ not dominated by $\{x_1, y_1\}$ is $P_2 \square P_{4k}$, and $2k$ vertices of $D-\{x_1, y_1\}$ must dominate $P_2 \square P_{4k}$. But $\gamma(P_2 \square P_{4k})=2k+1$ by Theorem \ref{domination 2n}, and we reach a contradiction. Second, note that no $D$ contains both $x_1$ and $x_2$: if $\{x_1, x_2\} \subseteq D$, then the part of $P_2\square C_{4k+3}$ not dominated by $\{x_1, x_2\}$, say $H^*$, must be dominated by $2k$ vertices. If we let $S^*=\{x_i, y_j \mid i \equiv 2, j \equiv 0 \pmod 4 \mbox{ and } 4 \le i, j \le 4k\}$, then $|S^*|=2k-1$ and the part of $P_2 \square C_{4k+3}$ not dominated by $S^* \cup \{x_1, x_2\}$ is a $P_4$, and one vertex of $D-(S^* \cup \{x_1, x_2\})$ must dominate $P_4$. But $\gamma(P_4)=2$, and we reach a contradiction. (Similarly, no $D$ contains both $x_1$ and $x_{4k+3}$.) Third, note that no $D$ contains both $x_1$ and $x_3$: if $\{x_1, x_3\} \subseteq D$, then a vertex in $N[y_2]=\{x_2, y_1, y_2, y_3\}$ must belong to $D$. Since $\{x_1, y_1\} \not\subseteq D$, $\{x_3, y_3\} \not\subseteq D$, and $\{x_1, x_2\} \not\subseteq D$, we need to consider $\{x_1, x_3, y_2\} \subseteq D$: since $|V(P_2 \square C_{4k+3})|=8k+6$ and $|N[\{x_1, y_2, x_3\}]|=8$, $2k-1$ vertices of $D-\{x_1, x_3, y_2\}$ must dominate $8k-2$ vertices, which is impossible since each vertex in $P_2 \square C_{4k+3}$ dominates four vertices. (Similarly, $\{x_1, x_{4k+2}\} \not\subseteq D$.) So, we only need to consider $D$ such that (i) $\{x_1, y_2\} \subseteq D$ (resp. $\{x_1, y_{4k+3}\} \subseteq D$) or (ii) no vertex in $N[x_1]$ is doubly dominated.  So suppose that $\{x_1, y_2\} \subseteq D$. Then the part of $P_2 \square C_{4k+3}$ that are not dominated by $\{x_1, y_2\}$, say $H''$, must be dominated by $2k$ vertices. Since $|V(P_2 \square C_{4k+3})|=8k+6$ and $|N[\{x_1, y_2\}]|=6$, $2k$ vertices of $D-\{x_1, y_2\}$ must dominate $H''$ with $|V(H'')|=8k$, and thus there exists at most one such $D$. Since $\{x_1, y_2\} \cup \{x_i, y_j \mid i \equiv 0, j \equiv 2 \pmod 4 \mbox{ and } 3 \le i, j \le 4k+3\}$ is a $\gamma$-set, if $\{x_1, y_2\} \subseteq D$, then there exists a unique such $D$. Similarly, there exists a unique $D$ containing both $x_1$ and $y_{4k+3}$. If no vertex in $N[x_1]$ is doubly dominated (i.e, $\{x_1, y_3, y_{4k+2}\} \subseteq D$), then there are $2k$ slots in which a pair of vertices of $D$ at distance two apart can be placed. Thus, $DV(x_1)=2+2k$ if $k \ge 1$, and $DV(x_1)=3$ if $k=0$.~\hfill
\end{proof}

As an immediate consequence of Theorem \ref{2Cn}, Observation \ref{Obs1}, Observation \ref{Obs2}, and the vertex-transitivity of $P_2 \square C_n$, we have the following. 

\begin{Crlr}
For $n \ge 3$, 
\begin{equation*}
\tau(P_2 \square C_n)
\left\{
\begin{array}{ll}
4 & \mbox{ if } n \equiv 0 \pmod 4\\
2n & \mbox{ if } n \equiv 1, 3 \pmod 4 \mbox{ and } n \neq 3\\
n(n+2) & \mbox{ if } n \equiv 2 \pmod 4 \mbox{ and } n \neq 6\\
9 & \mbox{ if } n=3\\
51 & \mbox{ if } n=6 \ .
\end{array} \right.
\end{equation*}
\end{Crlr} 


\section{Open Problems}

We end this paper with some open problems. One could ask the following questions.

\textbf{1.} In our terminology, Mynhardt \cite{Mynhardt} characterized vertices $v$ in a tree $T$ such that $DV(v)=\tau(T)$ or $DV(v)=0$. Can we describe vertices satisfying $DV(v)=k$ for $k \neq 0, \tau(T)$?

\textbf{2.} For $e \in E(G)$, can we find the bounds of $\tau(G-e)$ in terms of $\tau(G)$? And, for $v \in V(G-e)$, how does $DV_{G-e}(v)$ change in terms of $DV_G(v)$?
 
\textbf{3.} For $w\in V(G)$, can we find the bounds of $\tau(G-w)$ in terms of $\tau(G)$? And, for $v \in V(G-w)$, how does $DV_{G-w}(v)$ change in terms of $DV_{G}(v)$?

\textbf{4.} For a given graph $G$, can we characterize subgraphs $H \subseteq G$ satisfying $DV_H(v)=DV_G(v)$ for each vertex $v \in V(H)$?

In parallel with the idea of $\tau(G)$, the anonymous referee suggested the following questions.

\textbf{5.} Can we compute the \emph{number} of $ir$-sets (maximal irredundant sets of minimum cardinality), $\gamma$-sets (minimum dominating sets), $\gamma_t$-sets (minimum total dominating sets),  $i$-sets (minimum independent dominating sets), $\beta_0$-sets (maximum independent sets), $\Gamma$-sets (minimal dominating sets of maximum cardinality), $IR$-sets (maximum irredundant sets) in a graph $G$?\\

\textit{Acknowledgement.} The author thanks Cong X. Kang for suggesting the notion of domination value, as well as his helpful comments and suggestions. The author also thanks the referee for bringing to her attention the reference \cite{Slater}, and other helpful comments and suggestions.

\end{document}